\documentclass[12pt,reqno]{amsart}
\usepackage{amsthm, amscd, amsfonts, amssymb, graphicx, color}
\usepackage[bookmarksnumbered, colorlinks, plainpages,linkcolor=green,anchorcolor=black,citecolor=blue,urlcolor=black]{hyperref}
\usepackage{mathrsfs}
\usepackage{amssymb}
\usepackage{enumerate}
\usepackage{exscale}
\usepackage[all]{xy}
\usepackage{relsize}
\usepackage{pdfsync}
\numberwithin{equation}{section}

\makeatletter
\@namedef{subjclassname@2020}{%
  \textup{2020} Mathematics Subject Classification}
\makeatother

\usepackage[includemp,body={398pt,550pt},footskip=30pt,%
            marginparwidth=60pt,marginparsep=10pt]{geometry}

\newtheorem{theorem}{Theorem}[section]
\newtheorem{lemma}[theorem]{Lemma}
\newtheorem{proposition}[theorem]{Proposition}

\theoremstyle{definition}

\numberwithin{equation}{section}

\newcommand{\D}{\mathbb{D}}
\newcommand{\dd}{\mathcal{D}}
\newcommand{\h}{\mathscr{H}}
\newcommand{\C}{\mathbb{C}}
\newcommand{\T}{\mathbb{T}}
\newcommand{\Z}{\mathbb{Z}}
\newcommand{\R}{\mathcal{R}}
\newcommand{\E}{\mathbb{E}}
\newcommand{\Pp}{\mathbb{P}}

\begin{document}
\title[Littlewood-type theorem for random Dirichlet series]{A converse to Littlewood-type theorem for random Dirichlet series}
\date{November 13, 2026.
}
\author[J. Chen, M. He, and M. Wang]{Jiale Chen, Min He, and Maofa Wang}
\address{Jiale Chen, School of Mathematics and Statistics, Shaanxi Normal University, Xi'an 710119, China.}
\email{jialechen@snnu.edu.cn}

\address{Min He, School of Mathematics and Statistics, Wuhan University, Wuhan 430072, China.}
\email{hm.math@whu.edu.cn}

\address{Maofa Wang, School of Mathematics and Statistics, Wuhan University, Wuhan 430072, China.}
\email{mfwang.math@whu.edu.cn}

\thanks{Chen is supported by National Natural Science Foundation (No. 12501170), Young Talent Fund of Association for Science and Technology in Shaanxi (No. 20260506) and the Fundamental Research Funds for the Central Universities (No. GK202601003) of China. Wang is supported by National Natural Science Foundation (No. 12571145) of China.}



\begin{abstract}
  \noindent This article investigates random Dirichlet series within the framework of Hardy and mixed norm spaces of Dirichlet series. For randomization methods by sequences of non-constant independent and identically distributed random variables, we completely characterize when the random symbol space of the Hardy space $\mathscr{H}^p$ (respectively, the mixed norm space $\mathscr{H}^{p,q}_{\alpha}$) coincides with $\mathscr{H}^2$ (respectively, $\mathscr{H}^{2,q}_{\alpha}$). We also obtain some integrability results when Dirichlet series belong to the corresponding random symbol spaces. The same methods also yield a complete characterization for mixed norm spaces on the unit disk in the range $0<q<4$, closing the gap between previously known necessary and sufficient conditions.
\end{abstract}
\maketitle


\section{Introduction}
\allowdisplaybreaks[4]
Let $f(s)=\sum_{n=1}^{\infty}a_nn^{-s}$ be a Dirichlet series that converges at some point $s_0\in\C$, and let $\{X_n\}_{n\geq1}$ be a sequence of random variables on a probability space $(\Omega,\mathcal{F},\Pp)$. The random Dirichlet series associated with $f$ and $\{X_n\}_{n\geq1}$ is defined by
$$\R_X f(s):=\sum_{n=1}^{\infty}a_nX_nn^{-s}.$$
We simply write $\R f$ when the randomizing sequence is clear from the context. If the random variables $X_n$, $n\geq1$, are independent and identically distributed (i.i.d.), then by Kolmogorov's zero-one law, the abscissa of pointwise convergence of $\R f$ is almost surely a constant; see, for example, \cite[p. 43]{Ka85}.

The probabilistic theory of Dirichlet series has a long history, going back to works of Bohr--Jessen \cite{BJ}, Hartman \cite{Ha39}, Salem--Zygmund \cite{SZ}, among others; see also Kahane \cite{Ka85}. We refer the reader to \cite{BM} for the natural boundary problem, to \cite{A19,AFM} for investigations into the distribution of real zeros, and to \cite{CDS,CMST,KQ} for almost sure convergence of random Dirichlet series. In the modern function-space approach, Hardy and Bergman-type spaces of Dirichlet series are usually studied through the Bohr lift and their realization in analytic function spaces on the infinite torus. This viewpoint reveals deep connections between analytic number theory, infinite-dimensional harmonic analysis and probability. In this note, continuing the study in \cite{CGW}, we investigate random Dirichlet series in the context of Hardy and Bergman spaces of Dirichlet series.

We next recall the definitions of Hardy and Bergman spaces of Dirichlet series. Let $\dd$ be the space of Dirichlet series $\sum_{n=1}^{\infty}a_nn^{-s}$ that converge on some half-plane $\C_{\theta}:=\{s\in\C:\Re s>\theta\}$, $\theta\in\mathbb{R}$, and let $\mathcal{P}$ be the space of Dirichlet polynomials $\sum_{n=1}^Na_nn^{-s}$. The Hardy space $\mathscr H^2$, introduced by Hedenmalm, Lindqvist and Seip
\cite{HLS}, is the space of Dirichlet series
$f(s)=\sum_{n=1}^{\infty}a_nn^{-s}$ satisfying
\[
\|f\|_{\mathscr H^2}^2:=\sum_{n=1}^{\infty}|a_n|^2<\infty.
\]
For general $0<p<\infty$, Bayart \cite{Ba02} and Bondarenko et al. \cite{BBSS} defined Hardy spaces $\h^p$ by means of the Bohr lift. By the fundamental theorem of arithmetic, each
$n\in\mathbb N$ has a unique representation
$n=\mathfrak p^{\alpha(n)}=\mathfrak p_1^{\alpha_1(n)}\mathfrak p_2^{\alpha_2(n)}\cdots$,
where $\mathfrak{p}=\{\mathfrak{p}_k\}_{k\geq1}$ is the increasing sequence of prime numbers, and $\alpha(n)=(\alpha_1(n),\alpha_2(n),\ldots)$
is an eventually null sequence of non-negative integers. The Bohr lift of $f(s)=\sum_{n=1}^{\infty}a_nn^{-s}$ is the formal Fourier series
$$
\mathcal Bf(z):=\sum_{n=1}^{\infty}a_nz^{\alpha(n)},
\qquad z=(z_1,z_2,\ldots)\in\T^{\infty}.
$$
Here we use the multi-index notation $z^{\alpha(n)}:=z_1^{\alpha_1(n)}z_2^{\alpha_2(n)}\cdots$. For $0<p<\infty$, the Hardy space $\mathscr H^p$ is the completion of $\mathcal{P}$ with respect to the (quasi-)norm
\[
\|P\|_{\mathscr H^p}
:=\left(\int_{\T^{\infty}}|\mathcal BP(z)|^p\,dm_{\infty}(z)\right)^{1/p},\quad P\in\mathcal{P},
\]
where $m_{\infty}$ is the Haar measure on the countably infinite-dimensional torus $\T^{\infty}$ coinciding with the product of the normalized Lebesgue measure $m$ on the unit
circle $\T$. Then the Bohr lift $\mathcal{B}$ is an isometric isomorphism from $\h^p$ onto the analytic Hardy space $H^p(\T^{\infty})$. Recall that every element in $\h^p$ is a Dirichlet series that converges uniformly on $\C_{1/2+\epsilon}$ for any $\epsilon>0$, which means that $\h^p$ can be viewed as a space of analytic functions on the half-plane $\C_{1/2}$. Furthermore, $\h^q$ is contractively embedded into $\h^p$ for $0<p\leq q<\infty$; see \cite{Ba02,BBSS,DGMS,HLS,QQ} for more information about these spaces.

Given $0<p,q<\infty$ and $\alpha>-1$, the mixed norm space $\h^{p,q}_{\alpha}$ of Dirichlet series is defined as the completion of $\mathcal{P}$ with respect to the (quasi-)norm
$$\|P\|_{\h^{p,q}_{\alpha}}:
=\left(\int_0^{+\infty}\|P_{\sigma}\|_{\h^p}^q
d\mu_{\alpha}(\sigma)\right)^{1/q},\quad P\in\mathcal{P},$$
where $P_{\sigma}(\cdot):=P(\cdot+\sigma)$ denotes the horizontal translation of $P$, and
$$d\mu_{\alpha}(\sigma):=\frac{2^{\alpha+1}}{\Gamma(\alpha+1)}
\sigma^{\alpha}e^{-2\sigma}d\sigma.$$
These spaces were defined in \cite{CGW}. In particular, when $p=q$, we obtain the Bergman space $\mathcal{A}_{\alpha}^p:=\h^{p,p}_{\alpha}$ introduced in \cite{BL}. It is worth noting that $\h^{p,q}_{\alpha}$ is a space of Dirichlet series converging on $\C_{1/2}$ as well; see \cite[Theorem 2.3]{CGW}.

Recently, Chen, Guo and Wang \cite{CGW} investigated Littlewood-type theorems for random Dirichlet series in the setting of Hardy and mixed norm spaces, which exhibits some different phenomena compared with the case of random power series developed in \cite{CFL,Li26,Li30}. Let $\mathscr{X}\subset\dd$ be a (quasi-)Banach space of Dirichlet series containing all Dirichlet polynomials. Then, provided that the random variables $X_n$ are independent and the membership event is measurable, the Kolmogorov zero-one law (see \cite[Proposition II.2, p. 17]{LQ18}) ensures that for any $f\in\dd$,
$$\Pp(\R f\in\mathscr{X})\in\{0,1\}.$$
This prompts us to define the symbol space $\mathscr{X}_{\star}$ by
$$\mathscr{X}_{\star}:=\{f\in\mathcal{D}:\Pp(\R f\in\mathscr{X})=1\}.$$
For the three canonical randomization methods (Rademacher, Steinhaus, and Gaussian), the symbol spaces $(\h^p)_{\star}$ and $(\h^{p,q}_{\alpha})_{\star}$ were characterized in \cite{CGW}. More precisely, it was proved in \cite[Theorem 3.2]{CGW} that, if $\{X_n\}_{n\geq1}$ is i.i.d. and $X_1$ is a Rademacher, Steinhaus, or Gaussian $N(0,1)$ variable, then for $0<p<\infty$,
\begin{equation}\label{Hardy-Littlewood}
(\h^p)_{\star}=\h^2,
\end{equation}
and for $0<p,q<\infty$ and $\alpha>-1$,
\begin{equation}\label{mixednorm-Littlewood}
(\h^{p,q}_{\alpha})_{\star}=\h^{2,q}_{\alpha}.
\end{equation}

For random analytic functions in the unit disk, recent works have focused on converse of Littlewood-type theorem \cite{CCFL,CL}, that is, on finding conditions for general randomization methods such that the corresponding characterizations for symbol spaces remain the same as in the three canonical randomization methods. The purpose of the present paper is to develop the corresponding converse of Littlewood-type theorem for random Dirichlet series. More precisely, the central questions of this paper are the following.
\begin{itemize}
	\item[(A)] For i.i.d. random variable sequences $\{X_n\}_{n\geq1}$, how can we characterize the distributions of $X_1$ for which \eqref{Hardy-Littlewood} or \eqref{mixednorm-Littlewood} holds?
	\item[(B)] If \eqref{Hardy-Littlewood} or \eqref{mixednorm-Littlewood} holds, can we obtain stronger integrability information?
\end{itemize}

Random Dirichlet series differ from random power series in both convergence and function-space geometry. Their pointwise convergence, uniform convergence and boundedness are described by abscissae and half-planes rather than a radius and a disk of convergence. Under the Bohr lift, $n^{-s}$ corresponds to $z^{\alpha(n)}$, so coefficient randomization acts on monomials in countably many variables and the Hardy space norm is taken on $\T^{\infty}$ rather than $\T$. The frequencies $\log n$ are not equally spaced and reflect the multiplicative structure of the integers, in contrast to the additive grading of Taylor degrees. Moreover, mixed norm spaces involve horizontal translations, with multipliers $n^{-\sigma}$ and integration against $d\mu_\alpha(\sigma)$, instead of radial dilations and radial integration; this leads to logarithmic rather than degree-based dyadic blocks. Restricting to $n=2^k$ recovers a one-variable power series and transfers certain necessary conditions, but does not control arbitrary Dirichlet coefficients. Thus the analysis combines probabilistic estimates with the Bohr lift, Dirichlet polynomial approximation and logarithmic block estimates.

We now state our results. The first one answers question (A) in the Hardy space setting.

\begin{theorem}\label{main1}
Let $0<p<\infty$, and let $\{X_n\}_{n\geq 1}$ be a sequence of non-constant i.i.d. random variables.
\begin{enumerate}
	\item If $0<p\leq2$, then $(\h^p)_{\star}=\h^2$ if and only if $X_1\in L^2(\Omega)$.
	\item If $2<p<\infty$, then $(\h^p)_{\star}=\h^2$ if and only if $X_1\in L^2(\Omega)$ and $\E X_1=0$.
\end{enumerate}
\end{theorem}

The next result establishes some integrability under the almost sure membership condition.

\begin{theorem}\label{main2}
Let $0<p,q<\infty$ and let $\{X_n\}_{n\geq1}$ be a sequence of i.i.d. random variables. Suppose that $f\in(\h^p)_{\star}$ and $X_1\in L^t(\Omega)$ with $t=\max\{2,q\}$.
\begin{enumerate}
	\item If $0<p\leq 2$, then $\E\|\R f\|_{\h^p}^q<\infty$.
	\item If $2<p<\infty$ and $\E X_1=0$, then $\E\|\R f\|_{\h^p}^q<\infty$.
\end{enumerate}
\end{theorem}

The proofs of Theorems \ref{main1} and \ref{main2} are given in Section \ref{Hardy-case}.

We now turn to the mixed norm space setting. It is surprising that the inclusion $(\h^{p,q}_{\alpha})_{\star}\subset\h^{2,q}_{\alpha}$ holds for every sequence $\{X_n\}$ of non-constant i.i.d. random variables; see Theorem \ref{mixed-suff2}. Moreover, we have the following complete answer to question (A) in the mixed norm space setting.

\begin{theorem}\label{main3}
Let $0<p,q<\infty$, $\alpha>-1$, and let $\{X_n\}_{n\geq1}$ be a sequence of non-constant i.i.d. random variables. Let $t=\max\{2,q\}$.
\begin{enumerate}
	\item If $0<p\leq2$, then $(\h^{p,q}_{\alpha})_{\star}=\h^{2,q}_{\alpha}$ if and only if $X_1\in L^t(\Omega)$.
	\item If $2<p<\infty$, then $(\h^{p,q}_{\alpha})_{\star}=\h^{2,q}_{\alpha}$ if and only if $X_1\in L^t(\Omega)$ and $\E X_1=0$.
\end{enumerate}
\end{theorem}

We also have the following integrability result in the setting of mixed norm spaces.

\begin{theorem}\label{main6}
Let $0<p,q<\infty$, $\alpha>-1$, and let $\{X_n\}_{n\geq1}$ be a sequence of i.i.d. random variables. Suppose that $f\in(\h^{p,q}_{\alpha})_{\star}$ and $X_1\in L^t(\Omega)$ with $t=\max\{2,q\}$.
\begin{enumerate}
	\item If $0<p\leq 2$, then $\E\|\R f\|_{\h^{p,q}_{\alpha}}^q<\infty$.
	\item If $2<p<\infty$ and $\E X_1=0$, then $\E\|\R f\|_{\h^{p,q}_{\alpha}}^q<\infty$.
\end{enumerate}
\end{theorem}

The proof of Theorem \ref{main3} is given at the end of Section \ref{mixed-case}, using Theorems \ref{mixed-suff1}, \ref{mixed-nece} and \ref{mixed-suff2}. For Theorem \ref{main6}, the case where $X_1$ is constant a.s. is immediate; otherwise, apply Theorems \ref{mixed-suff2} and \ref{mixed-suff1}.

For mixed norm spaces on the unit disk, Chen and Liu \cite[Theorems 1--3]{CL} obtained a complete symbol space characterization for $q\geq4$, but separate necessary $L^q$ and sufficient $L^4$ conditions for $0<q<4$. Theorem \ref{main3} gives the complete characterization in the Dirichlet series setting throughout this range: the exact moment condition is $L^q$ for $2\leq q<4$ and $L^2$ for $0<q<2$, with zero mean additionally required when $p>2$. Our methods also apply to the unit disk. In Section \ref{disk-case}, Theorem \ref{disk-main} completes the characterization for $H^{p,q,u}(\D)$ in the range $0<q<4$, closing the gap between the necessary and sufficient conditions in \cite[Theorems 2 and 3]{CL}.

The proof combines coefficient block characterizations with two probabilistic lemmas. Bounded indicator truncation and the Paley--Zygmund inequality preserve divergence of weighted block sums without moment assumptions (Lemma \ref{block-retention}), and do not require decreasing coefficients as in \cite[Proposition 19]{CL}. Conversely, truncated strong laws and block averages detect an infinite second moment by producing coefficients with a finite deterministic block sum and an almost surely divergent randomized block sum (Lemma \ref{block-second-moment}). Together with symmetrization, the necessary $q$-th moment test and centering, these arguments give the exact threshold $\max\{2,q\}$ and the zero-mean requirement for $p>2$. Since the lemmas apply to both logarithmic blocks and Taylor degree blocks, they yield the complete characterization in both settings.

Throughout the paper, the notation $A\lesssim B$ (or $B\gtrsim A$) means that there exists a nonessential constant $C>0$ such that $A\leq CB$. If $A\lesssim B\lesssim A$, then we write $A\asymp B$. We use $\E$ to denote the expectation on the underlying probability space $(\Omega,\mathcal{F},\Pp)$.

\section{The Hardy space case}\label{Hardy-case}

In this section, we consider randomization in the context of Hardy spaces of Dirichlet series.

We begin with the following lemma, which relates random Dirichlet series to random power series in the Hardy space setting. Recall that for a power series $g(\xi)=\sum_{n=0}^{\infty}a_n\xi^n$ converging in the unit disk $\mathbb{D}$ and a sequence $\{Y_n\}_{n\geq0}$ of random variables, the random power series $\R_Yg$ associated with $g$ and $\{Y_n\}_{n\geq0}$ is defined by
$$\R_{Y}g(\xi):=\sum_{n=0}^{\infty}a_nY_n\xi^n.$$
Similarly to what was done for Dirichlet series, we define the symbol space $\big(H^p(\mathbb{D})\big)^{Y}_{\star}$ as
$$\big(H^p(\mathbb{D})\big)^{Y}_{\star}:=\left\{g(\xi)=\sum_{n=0}^{\infty}a_n\xi^n:\Pp\big(\R_Yg\in H^p(\mathbb{D})\big)=1\right\}$$
for the classical Hardy space $H^p(\mathbb{D})$ over the unit disk $\mathbb{D}$. The following lemma is an immediate consequence of the Bohr lift.

\begin{lemma}\label{necessity}
Let $0<p<\infty$ and let $\{X_n\}_{n\geq1}$ be a sequence of random variables. Denote $Y_k=X_{2^{k}}$ for $k\geq0$. If $(\h^p)_{\star}=\h^2$, then $(H^p(\D))^Y_{\star}=H^2(\D)$.
\end{lemma}

We also need the following lemma, which follows from \cite[Lemma 2.2]{CCFL}.

\begin{lemma}\label{h2-characterization}
Let $\{X_n\}_{n\geq1}$ be a sequence of i.i.d. random variables such that $\Pp(X_1\neq0)>0$. Then $(\h^2)_{\star}=\h^2$ if and only if $X_1\in L^2(\Omega)$.
\end{lemma}

Given a Dirichlet series $f(s)=\sum_{n=1}^{\infty}a_nn^{-s}$, we write
$$\R_{\epsilon}f(s):=\sum_{n=1}^{\infty}a_n\epsilon_nn^{-s}$$
for convenience, where $\{\epsilon_n\}_{n\geq1}$ is a sequence of i.i.d. Rademacher random variables on a probability space $(\Omega',\mathcal{F}',\Pp')$. Consequently,
$$\R_{\epsilon}\R f(s):=\sum_{n=1}^{\infty}a_n\epsilon_nX_nn^{-s}.$$
The following lemma considers the randomization with symmetry assumptions, i.e. $X_1$ and $-X_1$ have the same distribution.

\begin{lemma}\label{symmetric}
Let $\{X_n\}_{n\geq1}$ be a sequence of i.i.d. random variables. If $X_1$ is symmetric, then for $0<p<\infty$, $(\h^p)_{\star}=(\h^2)_{\star}$.
\end{lemma}
\begin{proof}
Assume that $X_1$ is symmetric, and let $\{\epsilon_n\}_{n\geq1}$ be a Rademacher sequence on a probability space $(\Omega',\mathcal{F}',\Pp')$ independent of $\{X_n\}_{n\geq1}$. By symmetry and independence, for almost every $\omega'\in\Omega'$, the sequences $\{\epsilon_n(\omega')X_n\}_{n\geq1}$ and $\{X_n\}_{n\geq1}$ have the same joint distribution. Consequently, for any $f(s)=\sum_{n=1}^{\infty}a_nn^{-s}\in\mathcal{D}$, $\R f\in\h^p$ a.s. if and only if for almost every $\omega'\in\Omega'$,
$$\Pp(\R_{\epsilon}\R f\in\h^p)=1,$$
which, by the Fubinization principle (see \cite[Proposition III.7, p. 26]{LQ18}), is equivalent to the assertion that, for almost every $\omega\in\Omega$,
\begin{equation}\label{omega}
\Pp'(\R_{\epsilon}\R f\in\h^p)=1.
\end{equation}
To apply \cite[Theorem 3.2]{CGW}, we first verify that this assertion implies $\R f\in\mathcal{D}$ a.s. Fix $\omega\in\Omega$ for which \eqref{omega} holds, and choose $\omega'\in\Omega'$ such that
$$\sum_{n=1}^{\infty}a_nX_n(\omega)\epsilon_n(\omega')n^{-s}\in\h^p.$$
This series converges in $\C_{1/2}$. Evaluating the convergent series at $s=1$, and using $|\epsilon_n(\omega')|=1$, we obtain a constant $C>0$ such that
$$|a_nX_n(\omega)|\leq Cn,\qquad n\geq1.$$
Hence $\sum_{n=1}^{\infty}a_nX_n(\omega)n^{-s}$ converges absolutely whenever $\Re s>2$, and thus belongs to $\mathcal{D}$. Applying \cite[Theorem 3.2]{CGW} to this Dirichlet series gives $\R f\in\h^2$ a.s. Conversely, if $\R f\in\h^2$ a.s., then \cite[Theorem 3.2]{CGW} gives \eqref{omega} for almost every $\omega$, and the preceding equivalences imply $\R f\in\h^p$ a.s. This proves the lemma.
\end{proof}

We are now ready to prove Theorem \ref{main1}.

\begin{proof}[Proof of Theorem \ref{main1}]
The necessity follows from Lemma \ref{necessity} and \cite[Theorem 1.1]{CCFL}. We consider the sufficiency.

(1) Suppose $0<p\leq2$ and $X_1\in L^2(\Omega)$. Then by Lemma \ref{h2-characterization},
$$\h^2=(\h^2)_{\star}\subset(\h^p)_{\star}.$$
Suppose now that $f(s)=\sum_{n=1}^{\infty}a_nn^{-s}$ belongs to $(\h^p)_{\star}$, i.e.
$$\sum_{n=1}^{\infty}a_nX_nn^{-s}\in \h^p\quad \text{a.s.}$$
Let $\{\widetilde{X}_n\}_{n\geq1}$ be an independent copy of $\{X_n\}_{n\geq1}$, defined on a product probability space. Then $X_n-\widetilde{X}_n$ is symmetric, and
$$\sum_{n=1}^{\infty}a_n(X_n-\widetilde{X}_n)n^{-s}\in\h^p\quad\text{a.s.}$$
Therefore, by Lemmas \ref{symmetric} and \ref{h2-characterization}, $f\in\h^2$, and so $(\h^p)_{\star}\subset\h^2$.

(2) Suppose $2<p<\infty$, $X_1\in L^2(\Omega)$ and $\E X_1=0$. Then by Lemma \ref{h2-characterization},
$$(\h^p)_{\star}\subset(\h^2)_{\star}=\h^2.$$
Suppose now that $f(s)=\sum_{n=1}^{\infty}a_nn^{-s}\in \h^2$. It is easy to see that for any $1\leq N\leq M<\infty$,
$$\E\left\|\sum_{n=N}^Ma_nX_nn^{-s}\right\|_{\h^p}^2
\leq\E\left(\sum_{n=N}^M|a_nX_n|\right)^2
\leq(M-N+1)\E|X_1|^2\sum_{n=1}^{\infty}|a_n|^2<\infty.$$
Since $\E X_1=0$, we can use \cite[Proposition II.13, p. 128]{LQ18} to obtain that
$$\E\left\|\sum_{n=N}^Ma_nX_nn^{-s}\right\|_{\h^p}^2
\leq\E\left\|\sum_{n=N}^Ma_n(X_n-\widetilde{X}_n)n^{-s}\right\|_{\h^p}^2,$$
where $\{\widetilde{X}_n\}_{n\geq1}$ is an independent copy of $\{X_n\}_{n\geq1}$. Let $\{\epsilon_n\}_{n\geq1}$ be a Rademacher sequence on a probability space $(\Omega',\mathcal{F}',\Pp')$ independent of $\{X_n-\widetilde{X}_n\}_{n\geq1}$. Then
$$\E\left\|\sum_{n=N}^Ma_n(X_n-\widetilde{X}_n)n^{-s}\right\|_{\h^p}^2
=\E\E_{\epsilon}\left\|\sum_{n=N}^Ma_n\epsilon_n(X_n-\widetilde{X}_n)n^{-s}\right\|_{\h^p}^2.$$
Since $2<p<\infty$, we know that $\h^p$ is of type $2$ (see \cite[Theorem IV.9, p. 188]{LQ18}). Consequently,
\begin{align*}
\E\left\|\sum_{n=N}^Ma_nX_nn^{-s}\right\|_{\h^p}^2
&\leq\E\E_{\epsilon}\left\|\sum_{n=N}^Ma_n\epsilon_n(X_n-\widetilde{X}_n)n^{-s}\right\|_{\h^p}^2\\
&\lesssim\E\sum_{n=N}^M\left\|a_n(X_n-\widetilde{X}_n)n^{-s}\right\|_{\h^p}^2\\
&=\sum_{n=N}^M|a_n|^2\E|X_n-\widetilde{X}_n|^2\\
&\leq4\E|X_1|^2\sum_{n=N}^{M}|a_n|^2\to0
\end{align*}
as $N,M\to\infty$. Therefore, the series $\sum_{n=1}^{\infty}a_nX_nn^{-s}$ converges in $L^2(\Omega;\h^p)$, which implies that $\R f\in\h^p$ a.s. That is, $f\in(\h^p)_{\star}$ and $\h^2\subset(\h^p)_{\star}$. The proof is complete.
\end{proof}

To prove Theorem \ref{main2}, we need the following elementary estimate.

\begin{lemma}\label{p2}
Let $0<p,q<\infty$ and let $\{X_n\}_{n\geq1}$ be a sequence of i.i.d. random variables. If $X_1\in L^q(\Omega)$ and is symmetric, then for any Dirichlet polynomial $f$,
$$\E\|\R f\|_{\h^p}^q\asymp\E\|\R f\|_{\h^2}^q$$
with implicit constants depending only on $p$ and $q$.
\end{lemma}
\begin{proof}
Suppose $f(s)=\sum_{n=1}^Na_nn^{-s}$. Since $X_1\in L^q(\Omega)$, it is easy to verify that $\E\|\R f\|_{\h^p}^q$ and $\E\|\R f\|_{\h^2}^q$ are both finite. Let $\{\epsilon_n\}_{n\geq1}$ be a Rademacher sequence on a probability space $(\Omega',\mathcal{F}',\Pp')$ independent of $\{X_n\}_{n\geq1}$. Then, since $X_1$ is symmetric,
$$\E\|\R f\|_{\h^p}^q=\E\E_{\epsilon}\|\R_{\epsilon}\R f\|_{\h^p}^q.$$
In the case $p\geq q$, we can use H\"older's inequality, Fubini's theorem and Khintchine's inequality to obtain that
\begin{align*}
\E\|\R f\|_{\h^p}^q
&=\E\E_{\epsilon}\left(\int_{\T^{\infty}}\left|\sum_{n=1}^Na_n\epsilon_nX_nz^{\alpha(n)}\right|^pdm_{\infty}(z)\right)^{q/p}\\
&\leq\E\left(\E_{\epsilon}\int_{\T^{\infty}}\left|\sum_{n=1}^Na_n\epsilon_nX_nz^{\alpha(n)}\right|^pdm_{\infty}(z)\right)^{q/p}\\
&\asymp\E\left(\sum_{n=1}^N|a_nX_n|^2\right)^{q/2}\\
&=\E\|\R f\|_{\h^2}^q,
\end{align*}
and
$$\E\|\R f\|_{\h^p}^q=\E\E_{\epsilon}\|\R_{\epsilon}\R f\|_{\h^p}^q
\geq\E\E_{\epsilon}\|\R_{\epsilon}\R f\|_{\h^q}^q\asymp\E\|\R f\|_{\h^2}^q.$$
The case $p<q$ is similar.
\end{proof}

We are now in a position to prove Theorem \ref{main2}.

\begin{proof}[Proof of Theorem \ref{main2}]
If $X_1$ is constant a.s., then there is nothing to prove. Assume that $X_1$ is non-constant, and write $f(s)=\sum_{n=1}^{\infty}a_nn^{-s}$.

(1) Since $f\in(\h^p)_{\star}$ and $X_1\in L^t(\Omega)$, by Theorem \ref{main1}, we have $f\in\h^2$, or equivalently, $\sum_{n=1}^{\infty}|a_n|^2<\infty$. Since $0<p\leq 2$, for any $1\leq N\leq M<\infty$, we have
$$\E\left\|\sum_{n=N}^Ma_nX_nn^{-s}\right\|_{\h^p}^q
\leq\E\left\|\sum_{n=N}^Ma_nX_nn^{-s}\right\|_{\h^2}^q=\E\left(\sum_{n=N}^M|a_nX_n|^2\right)^{q/2}.$$
If $0<q\leq2$, then
\begin{align*}
\E\left\|\sum_{n=N}^Ma_nX_nn^{-s}\right\|_{\h^p}^q
&\leq\E\left(\sum_{n=N}^M|a_nX_n|^2\right)^{q/2}\\
&\leq\left(\E\sum_{n=N}^M|a_nX_n|^2\right)^{q/2}\\
&=\left(\E|X_1|^2\right)^{q/2}\left(\sum_{n=N}^M|a_n|^2\right)^{q/2},
\end{align*}
and if $2<q<\infty$, then
\begin{align*}
\E\left\|\sum_{n=N}^Ma_nX_nn^{-s}\right\|_{\h^p}^q
&\leq\E\left(\sum_{n=N}^M|a_nX_n|^2\right)^{q/2}\\
&\leq\E\left(\sum_{n=N}^M|a_n|^2\right)^{\frac{q}{2}-1}\left(\sum_{n=N}^M|a_n|^2|X_n|^q\right)\\
&=\left(\E|X_1|^q\right)\left(\sum_{n=N}^{M}|a_n|^2\right)^{q/2}.
\end{align*}
Therefore, we have
$$\E\left\|\sum_{n=N}^Ma_nX_nn^{-s}\right\|_{\h^p}^q\to0$$
as $M,N\to\infty$. Consequently, the series $\sum_{n=1}^{\infty}a_nX_nn^{-s}$ converges in $L^q(\Omega;\h^p)$. That is, $\E\|\R f\|_{\h^p}^q<\infty$.

(2) Suppose that $2<p<\infty$ and $\E X_1=0$. By Theorem \ref{main1}, $f\in\h^2$. We first assume that $X_1$ is symmetric. Then for any $1\leq N\leq M<\infty$, Lemma \ref{p2} yields that
$$\E\left\|\sum_{n=N}^Ma_nX_nn^{-s}\right\|_{\h^p}^q\asymp\E\left\|\sum_{n=N}^Ma_nX_nn^{-s}\right\|_{\h^2}^q.$$
Hence the same argument as in (1) gives that $\E\|\R f\|_{\h^p}^q<\infty$.

We now consider the non-symmetric case. By the proof of Theorem \ref{main1}, the random partial sums of $\R f$ converge in $L^2(\Omega;\h^p)$. Consequently, $\mathcal Rf$ is Bochner integrable and has mean zero in $\mathscr H^p$ since $\E X_1=0$. Let $\{\widetilde{X}_n\}_{n\geq1}$ be an independent copy of $\{X_n\}_{n\geq1}$. Then we have
$$(\R f)_{\omega}(s)=\E_{\omega'}\left((\R f)_{\omega}(s)-(\widetilde{\R f})_{\omega'}(s)\right),$$
where $(\R f)_{\omega}(s)=\sum_{n=1}^{\infty}a_nX_n(\omega)n^{-s}$ and $(\widetilde{\R f})_{\omega'}(s)=\sum_{n=1}^{\infty}a_n\widetilde{X}_n(\omega')n^{-s}$. Combining the above equality with Minkowski's inequality yields that
$$\|(\R f)_{\omega}\|_{\h^p}^q
=\left\|\E_{\omega'}\left((\R f)_{\omega}-(\widetilde{\R f})_{\omega'}\right)\right\|_{\h^p}^q
\leq\left(\E_{\omega'}\left\|(\R f)_{\omega}-(\widetilde{\R f})_{\omega'}\right\|_{\h^p}\right)^q.$$
If $1\leq q<\infty$, then
\begin{equation}\label{omega-omega'}
\|(\R f)_{\omega}\|_{\h^p}^q\leq\E_{\omega'}\left\|(\R f)_{\omega}-(\widetilde{\R f})_{\omega'}\right\|_{\h^p}^q,
\end{equation}
and so
$$\E\|\R f\|_{\h^p}^q\leq\E_{\omega}\E_{\omega'}\left\|(\R f)_{\omega}-(\widetilde{\R f})_{\omega'}\right\|_{\h^p}^q<\infty$$
by the symmetric case. If $0<q<1$, then
\begin{align}\label{o-o'<1}
\E\|\R f\|_{\h^p}^q
&\leq\E_{\omega}\left(\E_{\omega'}\left\|(\R f)_{\omega}-(\widetilde{\R f})_{\omega'}\right\|_{\h^p}\right)^q\nonumber\\
&\leq\left(\E_{\omega}\E_{\omega'}\left\|(\R f)_{\omega}-(\widetilde{\R f})_{\omega'}\right\|_{\h^p}\right)^q.
\end{align}
Since $X_1\in L^2(\Omega)$, the same argument as in the symmetric case actually shows
$$\E_{\omega}\E_{\omega'}\left\|(\R f)_{\omega}-(\widetilde{\R f})_{\omega'}\right\|_{\h^p}<\infty.$$
Hence we also have $\E\|\R f\|_{\h^p}^q<\infty$ for $0<q<1$, which finishes the proof.
\end{proof}

\section{The mixed norm space case}\label{mixed-case}

In this section, we consider randomization in the context of mixed norm spaces of Dirichlet series. We begin with two probabilistic block lemmas, which will also be used in Section \ref{disk-case}.

\begin{lemma}\label{block-second-moment}
	Let $q>0$, let $\{B_k\}_{k\geq0}$ be pairwise disjoint finite sets of positive integers with $|B_k|\to\infty$, and let $w_k>0$. Suppose that $\{Y_n\}_{n\geq1}$ is a sequence of i.i.d. random variables with $\E|Y_1|^2=\infty$. Then there exists a deterministic non-negative sequence $\{a_n\}_{n\geq1}$, vanishing outside $\bigcup_{k\geq0}B_k$, such that
	$$a_n\leq w_k^{-1/q}\quad(n\in B_k),$$
	$$\sum_{k=0}^{\infty}w_k\left(\sum_{n\in B_k}a_n^2\right)^{q/2}<\infty,$$
	and
	$$\sum_{k=0}^{\infty}w_k\left(\sum_{n\in B_k}a_n^2|Y_n|^2\right)^{q/2}
	=\infty\quad\text{a.s.}$$
\end{lemma}
\begin{proof}
	For each positive integer $L$, the strong law of large numbers gives
	$$\frac{1}{m}\sum_{n=1}^{m}\min\{|Y_n|^2,L\}
	\longrightarrow\E\min\{|Y_1|^2,L\}\quad\text{a.s.}$$
	Taking the intersection of these probability-one events, we obtain, almost surely, for every positive integer $L$,
	$$\liminf_{m\to\infty}\frac1m\sum_{n=1}^m|Y_n|^2
	\geq\E\min\{|Y_1|^2,L\}.$$
	By the monotone convergence theorem, the right-hand side tends to infinity as $L\to\infty$. Hence
	$$\frac1m\sum_{n=1}^m|Y_n|^2\longrightarrow\infty\quad\text{a.s.}$$
	Consequently, for each $j\geq1$, there exists a positive integer $m_j$ such that
	\begin{equation}\label{m_j}
	\Pp\left(\frac1{m_j}\sum_{n=1}^{m_j}|Y_n|^2\geq2^{2j/q}\right)
	\geq1-2^{-j}.
	\end{equation}
	Since $|B_k|\to\infty$, we can choose a strictly increasing sequence $\{k_j\}_{j\geq1}$ of non-negative integers and sets $J_j\subset B_{k_j}$ with $|J_j|=m_j$. Define
	$$a_n=\begin{cases}
		m_j^{-1/2}w_{k_j}^{-1/q}2^{-j/q},&n\in J_j,\quad j\geq1,\\
		0,&n\notin\displaystyle\bigcup_{j\geq1}J_j.
	\end{cases}$$
	Then $a_n\leq w_k^{-1/q}$ for $n\in B_k$, and
	$$\sum_{k=0}^{\infty}w_k\left(\sum_{n\in B_k}a_n^2\right)^{q/2}
	=\sum_{j=1}^{\infty}2^{-j}<\infty.$$
	For $j\geq1$, let
	$$F_j=\left\{\frac1{m_j}\sum_{n\in J_j}|Y_n|^2\geq2^{2j/q}\right\}.$$
	Since the variables $Y_n$ are i.i.d., the inequality \eqref{m_j} yields $\Pp(F_j^c)\leq2^{-j}$. Then by the first Borel--Cantelli lemma (see \cite[Proposition I.2, p. 14]{LQ18}), $F_j$ occurs for all sufficiently large $j$ almost surely. On $F_j$,
	\begin{align*}
		w_{k_j}\left(\sum_{n\in B_{k_j}}a_n^2|Y_n|^2\right)^{q/2}
		=2^{-j}\left(\frac1{m_j}\sum_{n\in J_j}|Y_n|^2\right)^{q/2}\geq1.
	\end{align*}
	It follows that
	$$\sum_{k=0}^{\infty}w_k\left(\sum_{n\in B_k}a_n^2|Y_n|^2\right)^{q/2}
	=\infty\quad\text{a.s.},$$
	as required.
\end{proof}

\begin{lemma}\label{block-retention}
Let $r>0$, let $\{B_k\}_{k\geq0}$ be pairwise disjoint finite sets of non-negative integers, and let $w_k>0$. Suppose that $\{Y_n\}_{n\geq0}$ is a sequence of i.i.d. random variables with $\Pp(|Y_0|>0)>0$. For any deterministic non-negative sequence $\{c_n\}_{n\geq0}$, if
$$\sum_{k=0}^{\infty}w_k\left(\sum_{n\in B_k}c_n\right)^r=\infty,$$
then
$$\sum_{k=0}^{\infty}w_k\left(\sum_{n\in B_k}c_n|Y_n|^2\right)^r
=\infty\quad\text{a.s.}$$
\end{lemma}
\begin{proof}
Choose $\delta>0$ such that
$$\rho:=\Pp(|Y_0|\geq\delta)>0,$$
and put
$$A_k=\sum_{n\in B_k}c_n,\qquad
T_k=\sum_{n\in B_k}c_n\mathbf{1}_{\{|Y_n|\geq\delta\}},$$
where $\mathbf{1}_E$ denotes the indicator function of an event $E\in\mathcal{F}$. Then $0\leq T_k\leq A_k$, $\E T_k=\rho A_k$, and $\E T_k^2\leq A_k^2$. For every $k$ with $A_k>0$, the Paley--Zygmund inequality (see \cite[Proposition III.3, p. 23]{LQ18}) gives
$$\Pp\left(T_k\geq\frac{\rho A_k}{2}\right)
\geq\frac{(\E T_k)^2}{4\E T_k^2}
\geq\frac{\rho^2}{4}.$$
Let $E_k=\{T_k\geq\rho A_k/2\}$, with $E_k=\Omega$ when $A_k=0$. Then the events $E_k$ are independent, since the sets $B_k$ are pairwise disjoint, and $\Pp(E_k)\geq\rho^2/4$ for all $k\geq0$. Moreover,
\begin{equation}\label{cnYnlower}
\sum_{n\in B_k}c_n|Y_n|^2
\geq\delta^2T_k
\geq\frac{\delta^2\rho}{2}A_k\mathbf{1}_{E_k}.
\end{equation}
Set $d_k=w_kA_k^r$. Since $\sum_{k=0}^{\infty}d_k=\infty$, and noting that $1-e^{-x}\geq(1-e^{-1})\min\{x,1\}$ for $x\geq0$, we have $\sum_{k=0}^{\infty}(1-e^{-d_k})=\infty$. By independence,
\begin{align*}
\E\exp\left(-\sum_{k=0}^{N}d_k\mathbf{1}_{E_k}\right)
&=\prod_{k=0}^{N}\E\exp(-d_k\mathbf{1}_{E_k})\\
&=\prod_{k=0}^{N}\left(1-\Pp(E_k)(1-e^{-d_k})\right)\\
&\leq\prod_{k=0}^{N}\left(1-\frac{\rho^2}{4}(1-e^{-d_k})\right)\\
&\leq\exp\left(-\frac{\rho^2}{4}\sum_{k=0}^{N}(1-e^{-d_k})\right)
\to0.
\end{align*}
The dominated convergence theorem then yields
$$\E\exp\left(-\sum_{k=0}^{\infty}d_k\mathbf{1}_{E_k}\right)=0.$$
Hence $\sum_{k=0}^{\infty}d_k\mathbf{1}_{E_k}=\infty$ a.s., which together with \eqref{cnYnlower} implies
$$\sum_{k=0}^{\infty}w_k\left(\sum_{n\in B_k}c_n|Y_n|^2\right)^r
\geq\left(\frac{\delta^2\rho}{2}\right)^r
\sum_{k=0}^{\infty}d_k\mathbf{1}_{E_k}=\infty\quad\text{a.s.}$$
The proof is complete.
\end{proof}

We next establish a characterization of the spaces $\h^{2,q}_{\alpha}$ using logarithmic dyadic sums, which is of independent interest. For $k\geq0$, let $I_k$ be the logarithmic dyadic interval of positive integers defined by
$$I_k:=\left\{n:2^k\leq\log n<2^{k+1}\right\}.$$

\begin{theorem}\label{log-dyadic-characterization}
Let $0<q<\infty$ and $\alpha>-1$. Then $f(s)=\sum_{n=1}^{\infty}a_nn^{-s}$ belongs to $\h^{2,q}_{\alpha}$ if and only if
\begin{equation}\label{log-dyadic}
\sum_{k=0}^{\infty}2^{-k(\alpha+1)}
	\left(\sum_{n\in I_k}|a_n|^2\right)^{q/2}<\infty.
\end{equation}
\end{theorem}
\begin{proof}
By \cite[Equation (2.7)]{CGW}, we know that $f\in\h^{2,q}_{\alpha}$ if and only if
$$\int_0^{+\infty}\left(\sum_{n=1}^{\infty}|a_n|^2n^{-2\sigma}\right)^{q/2}d\mu_{\alpha}(\sigma)<\infty.$$
Suppose first that $f\in\h^{2,q}_{\alpha}$. Then by the definition of $I_k$, we have
\begin{align*}
\infty&>\int_{0}^{+\infty}\left(\sum_{n=1}^{\infty}
  |a_n|^2n^{-2\sigma}\right)^{q/2}d\mu_{\alpha}(\sigma)\\
&\geq\int_0^1\left(\sum_{k=0}^{\infty}\sum_{n\in I_k}
  |a_n|^2n^{-2\sigma}\right)^{q/2}d\mu_{\alpha}(\sigma)\\
&\geq\int_0^1\left(\sum_{k=0}^{\infty}\sum_{n\in I_k}
  |a_n|^2\left(e^{2^{k+1}}\right)^{-2\sigma}\right)^{q/2}
  d\mu_{\alpha}(\sigma)\\
&=\sum_{j=0}^{\infty}\int_{2^{-j-1}}^{2^{-j}}
  \left(\sum_{k=0}^{\infty}\sum_{n\in I_k}
  |a_n|^2\left(e^{2^{k+1}}\right)^{-2\sigma}\right)^{q/2}
  d\mu_{\alpha}(\sigma)\\
&\geq\sum_{j=0}^{\infty}\left(\sum_{k=0}^{\infty}\sum_{n\in I_k}
  |a_n|^2e^{-2^{k-j+2}}\right)^{q/2}
  \int_{2^{-j-1}}^{2^{-j}}d\mu_{\alpha}(\sigma)\\
&\asymp\sum_{j=0}^{\infty}2^{-j(\alpha+1)}
  \left(\sum_{k=0}^{\infty}\sum_{n\in I_k}
  |a_n|^2e^{-2^{k-j+2}}\right)^{q/2}\\
&\gtrsim\sum_{j=0}^{\infty}2^{-j(\alpha+1)}
  \left(\sum_{n\in I_j}|a_n|^2\right)^{q/2}.
\end{align*}
That is, \eqref{log-dyadic} holds.

Conversely, suppose that \eqref{log-dyadic} holds. Since every Dirichlet polynomial belongs to the space $\h^{2,q}_{\alpha}$, it is enough to show that
$$\int_0^{+\infty}\left(\sum_{n=3}^{\infty}|a_n|^2n^{-2\sigma}\right)^{q/2}d\mu_{\alpha}(\sigma)<\infty.$$
In the case $0<q\leq2$, we apply the definition of $I_k$ and H\"older's inequality to establish that
\begin{align*}
\int_0^{+\infty}\left(\sum_{n=3}^{\infty}
  |a_n|^2n^{-2\sigma}\right)^{q/2}d\mu_{\alpha}(\sigma)
&=\int_0^{+\infty}\left(\sum_{k=0}^{\infty}\sum_{n\in I_k}
  |a_n|^2n^{-2\sigma}\right)^{q/2}d\mu_{\alpha}(\sigma)\\
&\leq\int_0^{+\infty}\left(\sum_{k=0}^{\infty}\sum_{n\in I_k}
  |a_n|^2\left(e^{2^k}\right)^{-2\sigma}\right)^{q/2}
  d\mu_{\alpha}(\sigma)\\
&\leq\int_0^{+\infty}\sum_{k=0}^{\infty}
  \left(\sum_{n\in I_k}|a_n|^2\right)^{q/2}e^{-2^kq\sigma}
  d\mu_{\alpha}(\sigma)\\
&=\sum_{k=0}^{\infty}\left(\sum_{n\in I_k}|a_n|^2\right)^{q/2}
  \int_0^{+\infty}e^{-2^kq\sigma}d\mu_{\alpha}(\sigma).
\end{align*}
An elementary calculation gives that for $\kappa>1$,
$$\int_{0}^{+\infty}\kappa^{-\sigma}d\mu_{\alpha}(\sigma)
=\left(\frac{2}{2+\log\kappa}\right)^{\alpha+1}.$$
Therefore,
\begin{align*}
\int_0^{+\infty}\left(\sum_{n=3}^{\infty}
  |a_n|^2n^{-2\sigma}\right)^{q/2}d\mu_{\alpha}(\sigma)
&\leq\sum_{k=0}^{\infty}\left(\frac{2}{2+2^kq}\right)^{\alpha+1}
  \left(\sum_{n\in I_k}|a_n|^2\right)^{q/2}\\
&\lesssim\sum_{k=0}^{\infty}2^{-k(\alpha+1)}
  \left(\sum_{n\in I_k}|a_n|^2\right)^{q/2}<\infty.
\end{align*}
We next consider the case $2<q<\infty$. Write
$$b_k=2^{-\frac{2k(\alpha+1)}{q}}\sum_{n\in I_k}|a_n|^2$$
for $k\geq0$, and let $b_k=0$ for $k<0$. Then the condition \eqref{log-dyadic} is equivalent to $b:=\{b_k\}_{k\in\Z}\in l^{q/2}(\Z)$. For any $\sigma>0$,
\begin{equation*}
\sum_{n=3}^{\infty}|a_n|^2n^{-2\sigma}
\leq\sum_{k=0}^{\infty}e^{-2^{k+1}\sigma}\sum_{n\in I_k}|a_n|^2
=\sum_{k=0}^{\infty}b_k2^{\frac{2k(\alpha+1)}{q}}e^{-2^{k+1}\sigma},
\end{equation*}
which, combined with H\"older's inequality, yields that
\begin{align*}
\int_1^{+\infty}\left(\sum_{n=3}^{\infty}
  |a_n|^2n^{-2\sigma}\right)^{q/2}d\mu_{\alpha}(\sigma)
&\leq\int_1^{+\infty}\left(\sum_{k=0}^{\infty}
  b_k2^{\frac{2k(\alpha+1)}{q}}e^{-2^{k+1}\sigma}\right)^{q/2}
  d\mu_{\alpha}(\sigma)\\
&\leq\left(\sum_{k=0}^{\infty}
  b_k2^{\frac{2k(\alpha+1)}{q}}e^{-2^{k+1}}\right)^{q/2}\\
&\leq\left(\sum_{k=0}^{\infty}b_k^{q/2}\right)
  \left(\sum_{k=0}^{\infty}
  \left(2^{\frac{2k(\alpha+1)}{q}}e^{-2^{k+1}}\right)^{\frac{q}{q-2}}\right)^{\frac{q}{2}-1}\\
&<\infty.
\end{align*}
For the integral over $[0,1]$, we have
\begin{align*}
&\int_0^{1}\left(\sum_{n=3}^{\infty}
  |a_n|^2n^{-2\sigma}\right)^{q/2}d\mu_{\alpha}(\sigma)\\
&\ \leq\sum_{j=0}^{\infty}\int_{2^{-j-1}}^{2^{-j}}
  \left(\sum_{k=0}^{\infty}
  b_k2^{\frac{2k(\alpha+1)}{q}}e^{-2^{k+1}\sigma}\right)^{q/2}
  d\mu_{\alpha}(\sigma)\\
&\ \leq\sum_{j=0}^{\infty}\int_{2^{-j-1}}^{2^{-j}}
  \left(2^{\frac{2j(\alpha+1)}{q}}
  \sum_{k=0}^{\infty}b_k2^{\frac{2(k-j)(\alpha+1)}{q}}
  e^{-2^{k-j}}\right)^{q/2}d\mu_{\alpha}(\sigma)\\
&\ \asymp\sum_{j=0}^{\infty}
  \left(\sum_{k=0}^{\infty}b_k\lambda_{j-k}\right)^{q/2},
\end{align*}
where $\lambda_m=2^{-\frac{2m(\alpha+1)}{q}}e^{-2^{-m}}$ for $m\in\Z$. Note that
$$\sum_{m\in\Z}\lambda_m
\leq\sum_{m=0}^{\infty}2^{-\frac{2m(\alpha+1)}{q}}
+\sum_{m=1}^{\infty}2^{\frac{2m(\alpha+1)}{q}}e^{-2^m}<\infty.$$
That is, $\lambda:=\{\lambda_m\}_{m\in\Z}\in l^1(\Z)$. Consequently, noting that $q/2>1$, we can apply Minkowski's inequality for convolution to obtain that
\begin{align*}
\int_0^{1}\left(\sum_{n=3}^{\infty}
  |a_n|^2n^{-2\sigma}\right)^{q/2}d\mu_{\alpha}(\sigma)
&\lesssim\sum_{j=0}^{\infty}
  \left(\sum_{k=0}^{\infty}b_k\lambda_{j-k}\right)^{q/2}\\
&\leq\|b*\lambda\|_{l^{q/2}(\Z)}^{q/2}\\
&\leq\|b\|_{l^{q/2}(\Z)}^{q/2}\cdot\|\lambda\|_{l^1(\Z)}^{q/2}\\
&<\infty.
\end{align*}
The proof is complete.
\end{proof}

The following lemma will be useful in the sequel.

\begin{lemma}\label{sigma-integral}
Let $0<q<\infty$, $t=\max\{2,q\}$, and let $\{X_n\}_{n\geq1}$ be a sequence of i.i.d. random variables. If $X_1\in L^t(\Omega)$, then for any $f\in\h^2$ and any $\sigma\geq0$,
$$\E\|\R f_{\sigma}\|_{\h^2}^q
\leq\left(\E|X_1|^t\right)^{q/t}\|f_{\sigma}\|_{\h^2}^q.$$
\end{lemma}
\begin{proof}
Let $f(s)=\sum_{n=1}^{\infty}a_nn^{-s}$ belong to $\h^2$ and $\sigma\geq0$. In the case $0<q\leq2$, H\"older's inequality yields that
\begin{align*}
\E\|\R f_{\sigma}\|_{\h^2}^q
&=\E\left(\sum_{n=1}^{\infty}|a_nX_n|^2n^{-2\sigma}\right)^{q/2}\\
&\leq\left(\E\sum_{n=1}^{\infty}|a_nX_n|^2
  n^{-2\sigma}\right)^{q/2}\\
&=\left(\E|X_1|^2\right)^{q/2}\|f_{\sigma}\|_{\h^2}^q,
\end{align*}
and in the case $2<q<\infty$, Minkowski's inequality gives that
\begin{align*}
\E\|\R f_{\sigma}\|_{\h^2}^q
&=\E\left(\sum_{n=1}^{\infty}|a_nX_n|^2n^{-2\sigma}\right)^{q/2}\\
&\leq\left(\sum_{n=1}^{\infty}
  \left(\E|a_nX_n|^qn^{-q\sigma}\right)^{2/q}\right)^{q/2}\\
&=\left(\E|X_1|^q\right)\|f_{\sigma}\|_{\h^2}^q.\qedhere
\end{align*}
\end{proof}

The following lemma considers randomization with symmetry assumptions in the setting of mixed norm spaces, whose proof is similar to that of Lemma \ref{symmetric} and so is omitted.

\begin{lemma}\label{mixed-symmetric}
Let $0<p,q<\infty$, $\alpha>-1$, and let $\{X_n\}_{n\geq1}$ be a sequence of i.i.d. symmetric random variables. Then $(\h^{p,q}_{\alpha})_{\star}=(\h^{2,q}_{\alpha})_{\star}$.
\end{lemma}

We now establish sufficient conditions on the sequence $\{X_n\}_{n\geq1}$ ensuring that $\h^{2,q}_{\alpha}\subset(\h^{p,q}_{\alpha})_{\star}$.

\begin{theorem}\label{mixed-suff1}
Let $0<p,q<\infty$, $\alpha>-1$, and let $\{X_n\}_{n\geq1}$ be a sequence of i.i.d. random variables. Let $t=\max\{2,q\}$.
\begin{enumerate}
	\item If $0<p\leq 2$ and $X_1\in L^t(\Omega)$, then for any $f\in\h^{2,q}_{\alpha}$, $\E\|\R f\|_{\h^{p,q}_{\alpha}}^q<\infty$, and in particular, $\h^{2,q}_{\alpha}\subset(\h^{p,q}_{\alpha})_{\star}$.
	\item If $2<p<\infty$, $X_1\in L^t(\Omega)$ and $\E X_1=0$, then for any $f\in\h^{2,q}_{\alpha}$, $\E\|\R f\|_{\h^{p,q}_{\alpha}}^q<\infty$, and in particular, $\h^{2,q}_{\alpha}\subset(\h^{p,q}_{\alpha})_{\star}$.
\end{enumerate}
\end{theorem}
\begin{proof}
(1) Suppose $0<p\leq 2$ and $X_1\in L^t(\Omega)$. Fix $f(s)=\sum_{n=1}^{\infty}a_nn^{-s}\in\h^{2,q}_{\alpha}$. Then for any $1\leq N\leq M<\infty$, we deduce from Lemma \ref{sigma-integral} that
\begin{align*}
\E\left\|\sum_{n=N}^Ma_nX_nn^{-s}\right\|_{\h^{p,q}_{\alpha}}^q
&=\int_{0}^{+\infty}\E\left\|\sum_{n=N}^Ma_nX_nn^{-\sigma-s}\right\|_{\h^p}^qd\mu_{\alpha}(\sigma)\\
&\leq\int_{0}^{+\infty}\E\left\|\sum_{n=N}^Ma_nX_nn^{-\sigma-s}\right\|_{\h^2}^qd\mu_{\alpha}(\sigma)\\
&\leq\left(\E|X_1|^t\right)^{q/t}
  \int_{0}^{+\infty}\left\|\sum_{n=N}^Ma_nn^{-\sigma-s}\right\|_{\h^2}^qd\mu_{\alpha}(\sigma).
\end{align*}
Hence we have
$$\E\left\|\sum_{n=N}^Ma_nX_nn^{-s}\right\|_{\h^{p,q}_{\alpha}}^q\to0$$
as $N,M\to\infty$, which implies that $\E\|\R f\|_{\h^{p,q}_{\alpha}}^q<\infty$.

(2) Suppose $2<p<\infty$, $X_1\in L^t(\Omega)$ and $\E X_1=0$, and let $f(s)=\sum_{n=1}^{\infty}a_nn^{-s}$ belong to $\h^{2,q}_{\alpha}$. Let $\{\widetilde{X}_n\}_{n\geq1}$ be an independent copy of $\{X_n\}_{n\geq1}$. In the case $1\leq q<\infty$, by \eqref{omega-omega'}, Lemmas \ref{p2} and \ref{sigma-integral}, we have for any $1\leq N\leq M<\infty$,
\begin{align*}
\E\left\|\sum_{n=N}^Ma_nX_nn^{-s}\right\|_{\h^{p,q}_{\alpha}}^q
&=\int_{0}^{+\infty}\E\left\|\sum_{n=N}^Ma_nX_nn^{-\sigma-s}\right\|_{\h^p}^qd\mu_{\alpha}(\sigma)\\
&\leq\int_0^{+\infty}\E_{\omega\times\omega'}\left\|\sum_{n=N}^Ma_n(X_n-\widetilde{X}_n)n^{-\sigma-s}\right\|_{\h^p}^qd\mu_{\alpha}(\sigma)\\
&\asymp\int_0^{+\infty}\E_{\omega\times\omega'}\left\|\sum_{n=N}^Ma_n(X_n-\widetilde{X}_n)n^{-\sigma-s}\right\|_{\h^2}^qd\mu_{\alpha}(\sigma)\\
&\leq\left(\E_{\omega\times\omega'}|X_1-\widetilde{X}_1|^t\right)^{q/t}
  \int_0^{+\infty}\left\|\sum_{n=N}^Ma_nn^{-\sigma-s}\right\|_{\h^2}^qd\mu_{\alpha}(\sigma).
\end{align*} 
Since $f\in\h^{2,q}_{\alpha}$ and $X_1\in L^t(\Omega)$, we obtain that
$$\E\left\|\sum_{n=N}^Ma_nX_nn^{-s}\right\|_{\h^{p,q}_{\alpha}}^q\to0$$
as $N,M\to\infty$, which implies that $\E\|\R f\|_{\h^{p,q}_{\alpha}}^q<\infty$. In the case $0<q<1$, since $X_1\in L^2(\Omega)$, similarly we deduce from \eqref{o-o'<1} that
\begin{align*}
&\E\left\|\sum_{n=N}^Ma_nX_nn^{-s}\right\|_{\h^{p,q}_{\alpha}}^q\\
&\ =\int_{0}^{+\infty}\E\left\|\sum_{n=N}^Ma_nX_nn^{-\sigma-s}\right\|_{\h^p}^qd\mu_{\alpha}(\sigma)\\
&\ \leq\int_0^{+\infty}\left(\E_{\omega\times\omega'}\left\|\sum_{n=N}^Ma_n(X_n-\widetilde{X}_n)n^{-\sigma-s}\right\|_{\h^p}\right)^q
  d\mu_{\alpha}(\sigma)\\
&\ \asymp\int_0^{+\infty}\left(\E_{\omega\times\omega'}\left\|\sum_{n=N}^Ma_n(X_n-\widetilde{X}_n)n^{-\sigma-s}\right\|_{\h^2}\right)^q
  d\mu_{\alpha}(\sigma)\\
&\ \leq\left(\E_{\omega\times\omega'}|X_1-\widetilde{X}_1|^2\right)^{q/2}
  \int_0^{+\infty}\left\|\sum_{n=N}^Ma_nn^{-\sigma-s}\right\|_{\h^2}^qd\mu_{\alpha}(\sigma)\to0
\end{align*}
as $N,M\to\infty$, and consequently, $\E\|\R f\|_{\h^{p,q}_{\alpha}}^q<\infty$.
\end{proof}

The following theorem gives necessary conditions on $\{X_n\}_{n\geq1}$ for the inclusion $\h^{2,q}_{\alpha}\subset(\h^{p,q}_{\alpha})_{\star}$ to hold.

\begin{theorem}\label{mixed-nece}
Let $0<p,q<\infty$, $\alpha>-1$, and let $\{X_n\}_{n\geq1}$ be a sequence of non-constant i.i.d. random variables. Suppose that $\h^{2,q}_{\alpha}\subset(\h^{p,q}_{\alpha})_{\star}$. Then $X_1\in L^t(\Omega)$ with $t=\max\{2,q\}$. If, in addition, $2<p<\infty$, then $\E X_1=0$.
\end{theorem}

The proof of Theorem \ref{mixed-nece} is based on two lemmas. The first one is contained in the proof of \cite[Proposition 16]{CL}.

\begin{lemma}\label{L1-integrable}
Let $\{X_n\}_{n\geq1}$ be a sequence of i.i.d. random variables. If for any $\{c_n\}_{n\geq1}\in l^1$, $\sum_{n=1}^{\infty}|c_nX_n|<\infty$ a.s., then $X_1\in L^1(\Omega)$.
\end{lemma}

The following lemma gives a necessary second moment condition.

\begin{lemma}\label{p=2}
Let $0<q<\infty$, $\alpha>-1$, and let $\{X_n\}_{n\geq1}$ be a sequence of i.i.d. random variables. If $\h^{2,q}_{\alpha}\subset(\h^{2,q}_{\alpha})_{\star}$, then $X_1\in L^2(\Omega)$.
\end{lemma}
\begin{proof}
Assume, to the contrary, that $\E|X_1|^2=\infty$. Apply Lemma \ref{block-second-moment} with $B_k=I_k$, $w_k=2^{-k(\alpha+1)}$ and $Y_n=X_n$. The sets $I_k$ are pairwise disjoint and $|I_k|\to\infty$. The resulting coefficients satisfy
$$0\leq a_n\leq2^{k(\alpha+1)/q}\leq(\log n)^{(\alpha+1)/q}
\quad(n\in I_k),$$
and vanish outside these blocks. Thus $f(s)=\sum_{n=1}^{\infty}a_nn^{-s}$ converges absolutely on $\C_1$ and belongs to $\mathcal D$. The finite deterministic block sum and Theorem \ref{log-dyadic-characterization} give $f\in\h^{2,q}_{\alpha}$. The assumed inclusion then gives $\R f\in\h^{2,q}_{\alpha}$ a.s., whereas Lemma \ref{block-second-moment} makes its block sum infinite a.s. This contradicts Theorem \ref{log-dyadic-characterization} and proves the result.
\end{proof}

We are now ready to prove Theorem \ref{mixed-nece}.

\begin{proof}[Proof of Theorem \ref{mixed-nece}]
Let $f(s)=\sum_{n=1}^{\infty}a_n2^{-2^ns}$. Then by \cite[Lemma 4.2]{CGW}, $f\in\h^{p,q}_{\alpha}$ if and only if $\{2^{-(\alpha+1)n/q}a_n\}_{n\geq1}\in l^q$. Write $b_n=2^{-(\alpha+1)n/q}a_n$ for $n\geq1$. Since $\h^{2,q}_{\alpha}\subset(\h^{p,q}_{\alpha})_{\star}$, we have $\sum_{n=1}^{\infty}|b_nX_{2^{2^n}}|^q<+\infty$ a.s. whenever $\{b_n\}_{n\geq1}\in l^q$. Hence Lemma \ref{L1-integrable} yields $X_1\in L^q(\Omega)$.

We next show that $X_1\in L^2(\Omega)$. Let $\{\widetilde{X}_n\}_{n\geq1}$ be an independent copy of $\{X_n\}_{n\geq1}$, defined on a product probability space, and put $Z_n=X_n-\widetilde{X}_n$. For every $f\in\h^{2,q}_{\alpha}$, the assumed inclusion gives
$$\R_Z f=\R_X f-\R_{\widetilde X}f\in\h^{p,q}_{\alpha}
\quad\text{a.s.}$$
The sequence $\{Z_n\}_{n\geq1}$ is i.i.d. and symmetric. Applying Lemma \ref{mixed-symmetric} to this sequence, we obtain $\R_Z f\in\h^{2,q}_{\alpha}$ a.s. for every $f\in\h^{2,q}_{\alpha}$. Then Lemma \ref{p=2} yields
$$\E|X_1-\widetilde{X}_1|^2<\infty.$$
Let $\nu$ be the distribution of $X_1$. By independence and Fubini's theorem,
$$\int_{\C}\left(\int_{\C}|z-w|^2d\nu(z)\right)d\nu(w)=
\E|X_1-\widetilde{X}_1|^2<\infty.$$
Hence the inner integral is finite for $\nu$-almost every $w\in\C$. In particular, there exists $z_0\in\C$ such that
$$\int_{\C}|z-z_0|^2d\nu(z)<\infty.$$
Since $|z|^2\leq2|z-z_0|^2+2|z_0|^2$, we conclude that $X_1\in L^2(\Omega)$. Together with $X_1\in L^q(\Omega)$, this gives $X_1\in L^t(\Omega)$ for $t=\max\{2,q\}$.

Suppose now that $2<p<\infty$. Assume that $\E X_1\neq0$, and let $X_n'=X_n-\E X_n$. Then $X_1'\in L^t(\Omega)$ and $\E X_1'=0$. Since $2<p$, the proof of \cite[Theorem 1.6]{CGW} indicates that $\h^{p,q}_{\alpha}\subsetneq\h^{2,q}_{\alpha}$. Choose
$$g(s)=\sum_{n=1}^{\infty}c_nn^{-s}\in\h^{2,q}_{\alpha}\setminus\h^{p,q}_{\alpha}.$$
Then by Theorem \ref{mixed-suff1}, $\sum_{n=1}^{\infty}c_nX_n'n^{-s}\in \h^{p,q}_{\alpha}$ a.s., and so
$$\R g(s)=\sum_{n=1}^{\infty}c_nX_nn^{-s}=\sum_{n=1}^{\infty}c_nX_n'n^{-s}+\left(\E X_1\right)\sum_{n=1}^{\infty}c_nn^{-s}\notin\h^{p,q}_{\alpha}\quad \text{a.s.}$$
This contradicts the condition $\h^{2,q}_{\alpha}\subset(\h^{p,q}_{\alpha})_{\star}$ and completes the proof.
\end{proof}

We finally show that $(\h^{p,q}_{\alpha})_{\star}\subset\h^{2,q}_{\alpha}$ holds for any sequence $\{X_n\}_{n\geq1}$ of non-constant i.i.d. random variables.

\begin{theorem}\label{mixed-suff2}
Let $0<p,q<\infty$, $\alpha>-1$, and let $\{X_n\}_{n\geq1}$ be a sequence of non-constant i.i.d. random variables. Then $(\h^{p,q}_{\alpha})_{\star}\subset\h^{2,q}_{\alpha}$.
\end{theorem}
\begin{proof}
Assume first that $X_1$ is symmetric. By Lemma \ref{mixed-symmetric}, it suffices to consider $f(s)=\sum_{n=1}^{\infty}a_nn^{-s}\in(\h^{2,q}_{\alpha})_{\star}$. If $f\notin\h^{2,q}_{\alpha}$, Theorem \ref{log-dyadic-characterization} gives
$$\sum_{k=0}^{\infty}2^{-k(\alpha+1)}
\left(\sum_{n\in I_k}|a_n|^2\right)^{q/2}=\infty.$$
Apply Lemma \ref{block-retention} to the sequence $Y_n=X_{n+1}$ with
$$B_k=\{n-1:n\in I_k\},\qquad w_k=2^{-k(\alpha+1)},
\qquad c_n=|a_{n+1}|^2,\qquad r=q/2.$$
Since $X_1$ is non-constant, $\Pp(|Y_0|>0)>0$. The lemma yields
$$\sum_{k=0}^{\infty}2^{-k(\alpha+1)}
\left(\sum_{n\in I_k}|a_nX_n|^2\right)^{q/2}=\infty
\quad\text{a.s.}$$
This contradicts Theorem \ref{log-dyadic-characterization} applied to $\R f\in\h^{2,q}_{\alpha}$ a.s. Thus $f\in\h^{2,q}_{\alpha}$.

For the general case, let $\{\widetilde{X}_n\}_{n\geq1}$ be an independent copy of $\{X_n\}_{n\geq1}$. If $f(s)=\sum_{n=1}^{\infty}a_nn^{-s}\in(\h^{p,q}_{\alpha})_{\star}$, then
$$\sum_{n=1}^{\infty}a_nX_nn^{-s}\in\h^{p,q}_{\alpha}
\quad\text{and}\quad
\sum_{n=1}^{\infty}a_n\widetilde{X}_nn^{-s}\in\h^{p,q}_{\alpha}
\quad\text{a.s.}$$
Consequently,
$$\sum_{n=1}^{\infty}a_n(X_n-\widetilde{X}_n)n^{-s}
\in\h^{p,q}_{\alpha}\quad\text{a.s.}$$
Applying the symmetric case to the sequence $\{X_n-\widetilde{X}_n\}_{n\geq1}$ gives $f\in\h^{2,q}_{\alpha}$. The proof is complete.
\end{proof}

We are now in a position to prove Theorem \ref{main3}.

\begin{proof}[Proof of Theorem \ref{main3}]
Let $t=\max\{2,q\}$. Suppose first that $(\h^{p,q}_{\alpha})_{\star}=\h^{2,q}_{\alpha}$. Then $\h^{2,q}_{\alpha}\subset(\h^{p,q}_{\alpha})_{\star}$, and Theorem \ref{mixed-nece} gives $X_1\in L^t(\Omega)$. If $2<p<\infty$, the same theorem also gives $\E X_1=0$. This proves the necessity in both (1) and (2).

Conversely, suppose that $X_1\in L^t(\Omega)$ and that either $0<p\leq2$, or $2<p<\infty$ and $\E X_1=0$. By Theorem \ref{mixed-suff1}, we have $\h^{2,q}_{\alpha}\subset(\h^{p,q}_{\alpha})_{\star}$. Since $X_1$ is non-constant, Theorem \ref{mixed-suff2} gives the reverse inclusion. Therefore $(\h^{p,q}_{\alpha})_{\star}=\h^{2,q}_{\alpha}$, as required.
\end{proof}

\section{The unit disk case}\label{disk-case}

In this section, we apply the block arguments from Section \ref{mixed-case} to random power series on the unit disk. We use the notation of Chen and Liu \cite{CL}. Let $H(\D)$ be the space of analytic functions on the unit disk $\D$. For $0<p,q,u<\infty$, let $H^{p,q,u}(\D)$ be the space of analytic functions on $\D$ such that
$$\|f\|_{p,q,u}^q:=\int_0^1(1-r)^{qu-1}M_p(r,f)^qdr<\infty,$$
where
$$M_p(r,f):=\left(\int_{\T}|f(r\zeta)|^pdm(\zeta)\right)^{1/p}.$$
Given a sequence $\{X_n\}_{n\geq0}$ of i.i.d. random variables and $f(z)=\sum_{n=0}^{\infty}a_nz^n\in H(\D)$, write
$$\R_X f(z):=\sum_{n=0}^{\infty}a_nX_nz^n,$$
and define
$$\big(H^{p,q,u}(\D)\big)_{\star}
:=\{f\in H(\D):\Pp(\R_X f\in H^{p,q,u}(\D))=1\}.$$
As before, we write $\R f$ when the randomizing sequence is clear.

Chen and Liu \cite[Theorems 1--3]{CL} characterized when the equality
$$\big(H^{p,q,u}(\D)\big)_{\star}=H^{2,q,u}(\D)$$
holds for $q\geq4$, and obtained necessary $L^q$ and sufficient $L^4$ conditions for $0<q<4$. Using our probabilistic block arguments, we show that the second moment condition in the sufficient hypotheses of \cite[Theorem 14]{CL} is also necessary when $0<q<2$, and that the fourth moment assumption in \cite[Theorem 18]{CL} can be removed without imposing monotonicity on the Taylor coefficients. This leads to a complete characterization in the range $0<q<4$, closing the gap between the necessary and sufficient conditions in \cite[Theorems 2 and 3]{CL}.

\begin{theorem}\label{disk-main}
Let $0<p,q,u<\infty$, and let $\{X_n\}_{n\geq0}$ be a sequence of non-constant i.i.d. random variables. Let $t=\max\{2,q\}$.
\begin{enumerate}
\item If $0<p\leq2$, then $\big(H^{p,q,u}(\D)\big)_{\star}=H^{2,q,u}(\D)$ if and only if $X_0\in L^t(\Omega)$.
\item If $2<p<\infty$, then $\big(H^{p,q,u}(\D)\big)_{\star}=H^{2,q,u}(\D)$ if and only if $X_0\in L^t(\Omega)$ and $\E X_0=0$.
\end{enumerate}
\end{theorem}

For $k\geq0$, let
$$D_k:=\{n\geq1:2^k\leq n<2^{k+1}\}.$$
By \cite[Theorem 4.2.8]{JVA} (see also \cite[Lemma 17]{CL}), a function $f(z)=\sum_{n=0}^{\infty}a_nz^n\in H(\D)$ belongs to $H^{2,q,u}(\D)$ if and only if
\begin{equation}\label{disk-block}
\sum_{k=0}^{\infty}2^{-kqu}
\left(\sum_{n\in D_k}|a_n|^2\right)^{q/2}<\infty.
\end{equation}
We shall apply Lemmas \ref{block-second-moment} and \ref{block-retention} with the dyadic blocks $D_k$ and weights $2^{-kqu}$.

We next record the symmetry reduction for the disk setting.

\begin{lemma}\label{disk-symmetric}
Let $0<p,q,u<\infty$, and let $\{X_n\}_{n\geq0}$ be a sequence of i.i.d. symmetric random variables. Then
$$\big(H^{p,q,u}(\D)\big)_{\star}
=\big(H^{2,q,u}(\D)\big)_{\star}.$$
\end{lemma}
\begin{proof}
Let $\{\epsilon_n\}_{n\geq0}$ be an independent Rademacher sequence on $(\Omega',\mathcal F',\Pp')$. For any $f\in H(\D)$, symmetry and the Fubinization principle show that $\R_Xf\in H^{p,q,u}(\D)$ a.s. if and only if, for almost every $\omega\in\Omega$,
\begin{equation}\label{disk-sign-event}
\Pp'\left(\sum_{n=0}^{\infty}a_nX_n(\omega)\epsilon_nz^n
\in H^{p,q,u}(\D)\right)=1.
\end{equation}
For each such $\omega$, choose one sign sequence in this event. Multiplication of coefficients by signs does not change the radius of convergence, so the power series
$$\sum_{n=0}^{\infty}a_nX_n(\omega)z^n$$
also converges on $\D$. Then \cite[Theorem 6]{CFL}, recalled in \cite[Equation (1)]{CL}, applies and shows that \eqref{disk-sign-event} yields
$$\sum_{n=0}^{\infty}a_nX_n(\omega)z^n\in H^{2,q,u}(\D).$$
Conversely, this membership implies \eqref{disk-sign-event} by the same characterization. The preceding equivalence proves the result.
\end{proof}

The next proposition is the disk counterpart of Theorem \ref{mixed-suff2}.

\begin{proposition}\label{disk-reverse}
Let $0<p,q,u<\infty$, and let $\{X_n\}_{n\geq0}$ be a sequence of non-constant i.i.d. random variables. Then
$$\big(H^{p,q,u}(\D)\big)_{\star}\subset H^{2,q,u}(\D).$$
\end{proposition}
\begin{proof}
Suppose first that $X_0$ is symmetric. By Lemma \ref{disk-symmetric}, it suffices to consider $f(z)=\sum_{n=0}^{\infty}a_nz^n\in\big(H^{2,q,u}(\D)\big)_{\star}$. If $f\notin H^{2,q,u}(\D)$, then \eqref{disk-block} fails. Apply Lemma \ref{block-retention} with
$$B_k=D_k,\qquad w_k=2^{-kqu},\qquad c_n=|a_n|^2,
\qquad r=q/2.$$
It follows that
$$\sum_{k=0}^{\infty}2^{-kqu}
\left(\sum_{n\in D_k}|a_nX_n|^2\right)^{q/2}=\infty
\quad\text{a.s.}$$
This contradicts \eqref{disk-block} applied to $\R_Xf\in H^{2,q,u}(\D)$ a.s.

For the general case, let $\{\widetilde X_n\}_{n\geq0}$ be an independent copy of $\{X_n\}_{n\geq0}$. If $f\in\big(H^{p,q,u}(\D)\big)_{\star}$, then
$$\R_{X-\widetilde X}f=\R_Xf-\R_{\widetilde X}f
\in H^{p,q,u}(\D)\quad\text{a.s.}$$
The differences $X_n-\widetilde X_n$ are i.i.d., symmetric and non-constant. The symmetric case gives $f\in H^{2,q,u}(\D)$.
\end{proof}

The next lemma is the disk counterpart of Lemma \ref{p=2}.

\begin{lemma}\label{disk-second-moment}
Let $0<q,u<\infty$, and let $\{X_n\}_{n\geq0}$ be a sequence of i.i.d. random variables. If
$$H^{2,q,u}(\D)\subset\big(H^{2,q,u}(\D)\big)_{\star},$$
then $X_0\in L^2(\Omega)$.
\end{lemma}
\begin{proof}
Assume, to the contrary, that $\E|X_0|^2=\infty$. Apply Lemma \ref{block-second-moment} to $\{X_n\}_{n\geq1}$ with $B_k=D_k$ and $w_k=2^{-kqu}$. Its hypotheses hold because $|D_k|=2^k\to\infty$ and $X_1$ has the same distribution as $X_0$. Set $a_0=0$ and let $a_n$, $n\geq1$, be the resulting coefficients. For $n\in D_k$,
$$0\leq a_n\leq2^{ku}\leq n^u.$$
Thus $f(z)=\sum_{n=0}^{\infty}a_nz^n$ has radius of convergence at least $1$, and belongs to $H(\D)$. By \eqref{disk-block}, the finite deterministic block sum gives $f\in H^{2,q,u}(\D)$. The assumed inclusion implies $\R_Xf\in H^{2,q,u}(\D)$ a.s., contradicting the almost sure divergence of its block sum supplied by Lemma \ref{block-second-moment}. This proves the lemma.
\end{proof}

We are now ready to prove Theorem \ref{disk-main}.

\begin{proof}[Proof of Theorem \ref{disk-main}]
We first prove the necessity under the weaker assumption
$$H^{2,q,u}(\D)\subset\big(H^{p,q,u}(\D)\big)_{\star}.$$
By \cite[Proposition 16]{CL}, this inclusion implies $X_0\in L^q(\Omega)$. To obtain the second moment condition, let $\{\widetilde X_n\}_{n\geq0}$ be an independent copy of $\{X_n\}_{n\geq0}$ and put $Z_n=X_n-\widetilde X_n$. For every $f\in H^{2,q,u}(\D)$, the assumed inclusion gives $\R_Zf\in H^{p,q,u}(\D)$ a.s. By Lemma \ref{disk-symmetric}, $\R_Zf\in H^{2,q,u}(\D)$ a.s. Applying Lemma \ref{disk-second-moment} to $\{Z_n\}_{n\geq0}$, we obtain
$$\E|X_0-\widetilde X_0|^2<\infty,$$
which gives $X_0\in L^2(\Omega)$ as in the proof of Theorem \ref{mixed-nece}. Together with $X_0\in L^q(\Omega)$, this yields $X_0\in L^t(\Omega)$ for $t=\max\{2,q\}$.

Suppose now that $p>2$. We verify that $H^{p,q,u}(\D)$ is a proper subspace of $H^{2,q,u}(\D)$. Choose
$$\max\left\{\frac12,u+\frac1p\right\}<\gamma<u+\frac12,$$
and let $g(z)=(1-z)^{-\gamma}$, using the analytic branch on $\D$ with $g(0)=1$. For $v=2,p$, we have $\gamma v>1$, and
$$M_v(r,g)^v=\int_{\T}|1-r\zeta|^{-\gamma v}dm(\zeta)
\asymp(1-r)^{1-\gamma v},\qquad r\to1^-.$$
The last step follows from the classical Forelli--Rudin estimate. Hence $g\in H^{v,q,u}(\D)$ if and only if $\gamma<u+1/v$, for these two values of $v$. Our choice of $\gamma$ gives
$$g\in H^{2,q,u}(\D)\setminus H^{p,q,u}(\D).$$
If $\E X_0\neq0$, then $X'_n=X_n-\E X_n$ satisfies $X'_0\in L^t(\Omega)$ and $\E X'_0=0$. By \cite[Theorem 14]{CL}, $\R_{X'}g\in H^{p,q,u}(\D)$ a.s. The assumed inclusion also gives $\R_Xg\in H^{p,q,u}(\D)$ a.s. Subtracting, we obtain
$$(\E X_0)g=\R_Xg-\R_{X'}g\in H^{p,q,u}(\D),$$
a contradiction. Therefore $\E X_0=0$ whenever $p>2$.

Conversely, suppose that $X_0\in L^t(\Omega)$ and that either $0<p\leq2$, or $p>2$ and $\E X_0=0$. By \cite[Theorem 14]{CL},
$$H^{2,q,u}(\D)\subset\big(H^{p,q,u}(\D)\big)_{\star}.$$
Proposition \ref{disk-reverse} gives the reverse inclusion, and the proof is complete.
\end{proof}

\medskip





\end{document}